\documentclass[12pt,reqno]{amsart}

\usepackage{amssymb}
\usepackage{xypic}

\newtheorem{theorem}{Theorem}[section]
\newtheorem{proposition}[theorem]{Proposition}
\newtheorem{lemma}[theorem]{Lemma}
\newtheorem{corollary}[theorem]{Corollary}

\theoremstyle{remark}
\newtheorem{remark}[theorem]{Remark}

\numberwithin{equation}{section}

\begin{document}

\title[Direct image of parabolic line bundles]{Direct image of parabolic line bundles}

\author[R. Auffarth]{Robert Auffarth}

\address{Departamento de Matem\'aticas, Facultad de
Ciencias, Universidad de Chile, Santiago, Chile}

\email{rfauffar@uchile.cl}

\author[I. Biswas]{Indranil Biswas}

\address{School of Mathematics, Tata Institute of Fundamental Research,
Homi Bhabha Road, Mumbai 400005, India}

\email{indranil@math.tifr.res.in}



\date{}

\begin{abstract}
Given a vector bundle $E$, on an irreducible projective variety $X$, we give a 
necessary and sufficient criterion for $E$ to be a direct image of a line bundle 
under a surjective \'etale morphism. The criterion in question is the existence of a Cartan 
subalgebra bundle of the endomorphism bundle $\text{End}(E)$. As a corollary, a 
criterion is obtained for $E$ to be the direct image of the structure sheaf under an 
\'etale morphism. The direct image of a parabolic line bundle under any ramified 
covering map has a natural parabolic structure. Given a parabolic vector bundle, we 
give a similar criterion for it to be the direct image of a parabolic line bundle under a 
ramified covering map.
\end{abstract}

\maketitle

\section{Introduction}\label{sec1}

This work was inspired by \cite{DP}, and also by
\cite{Bo1}, \cite{Bo2}. In \cite{DP}, the authors address the following question:
Given a vector bundle $E$ on a smooth projective curve $X$ over a field
of characteristic zero, is there a branched covering of $X$ and a line bundle $L$ on $X$ 
such that $E\otimes L$ is the direct image of the structure sheaf under
the covering map? They answer this question affirmatively.

Let $X$ and $Y$ be smooth projective curves defined over an algebraically closed 
field $k$, where $X$ is irreducible but $Y$ need not be, and let $f\, :\, Y\, 
\longrightarrow\, X$ be a finite separable morphism. Then for any parabolic line 
bundle $L_*$ on $Y$, the direct image $f_*L_*$ has a natural parabolic structure.

Here we address the following questions:

Given a parabolic vector bundle $E_*$ on $X$, when there is a pair $(Y,\, f)$ as
above such that
\begin{enumerate}
\item $E_*$ is isomorphic to the parabolic vector bundle $f_*{\mathcal O}_Y$
(the parabolic structure on ${\mathcal O}_Y$ is the trivial one), and more generally,

\item $E_*$ is isomorphic to the parabolic vector bundle $f_*L_*$, where $L_*$ is a
parabolic line bundle on $Y$?
\end{enumerate}
Under the assumption that the characteristic of $k$ is zero, the first question is answered
in Corollary \ref{cor3} and the second question is answered in Theorem \ref{thm2}. When the
characteristic of $k$ is positive, these results remain valid if the parabolic structure on
$E_*$ is tame; see Section \ref{sc3p}.

To understand the criteria in Corollary \ref{cor3} and Theorem \ref{thm2}
the key step is to consider the \'etale case. More precisely, consider the
following questions:
\begin{enumerate}
\item Given a vector bundle $E$ on $X$, when there is an \'etale covering
$f\, :\, Y\, \longrightarrow\, X$,
and a line bundle $L$ on $Y$, such that $f_*L\,=\, E$.

\item With $X$ and $E$ as above, when there is an \'etale covering
$f\, :\, Y\, \longrightarrow\, X$ such that $f_*{\mathcal O}_Y\,=\, E$.
\end{enumerate}
We prove the following (see Theorem \ref{thm1} and Corollary \ref{cor2}):

\begin{theorem}\label{thm0}
Let $X$ be an irreducible projective variety defined over an algebraically closed field $k$.
Given a vector bundle $E$ on $X$ of rank $d$, the following two are equivalent:
\begin{enumerate}
\item There is an \'etale covering of degree $d$
$$
f\, :\, Y\, \longrightarrow\, X
$$
and a line bundle $L$ on $Y$, such that $f_*L\,=\, E$.

\item There is a subbundle ${\mathcal A}\, \subset\, {\rm End}(E)$ of rank $d$ such that for each
closed point
$x\in\, X$, the subspace ${\mathcal A}_x\, \subset\, {\rm End}(E)_x\,=\, {\rm End}(E_x)$ is a Cartan
subalgebra.
\end{enumerate}
If there is a subbundle ${\mathcal A}\, \subset\, {\rm End}(E)$ as in the second statement,
then $(Y,\, f)$ in the first statement can be so chosen that
${\mathcal A}\,=\, f_*{\mathcal O}_Y$. In that case, the number of connected
components of the scheme $Y$ coincides with $\dim H^0(X,\, {\mathcal A})$.
\end{theorem}

\begin{corollary}\label{c2}
Given $X$ and $E$ as in Theorem \ref{thm0}, the following two are equivalent:
\begin{enumerate}
\item There is an \'etale covering
$$
f\, :\, Y\, \longrightarrow\, X
$$
such that $f_*{\mathcal O}_Y\,=\, E$.

\item There is a fiberwise injective homomorphism $\alpha\, :\, E\, \longrightarrow\, {\rm End}(E)$
such that for each closed point
$x\in\, X$, the subspace $\alpha(E_x)\, \subset\, {\rm End}(E)_x\,=\, {\rm End}(E_x)$ is a Cartan
subalgebra.
\end{enumerate}
When these hold, the number of connected components of
the scheme $Y$ coincides with $\dim H^0(X,\, E)$.
\end{corollary}

In Section \ref{se5} we describe, in terms of the above criterion, when an \'etale
cover factors through another given \'etale covering.

\section{Direct image of line bundles by \'etale coverings}

\subsection{Construction of homomorphism of vector bundles from direct image}

Let $k$ be an algebraically closed field. Let $X$ be an irreducible projective variety
defined over $k$. Take any pair $(Y,\, f)$, where $Y$ is a projective scheme and
$$
f\, :\, Y\, \longrightarrow\, X
$$
is an \'etale covering of degree $d$. We do not assume that $Y$ is connected. Take any line bundle
$L$ over $Y$.
The direct image
\begin{equation}\label{e2}
E\,:=\, f_*L\, \longrightarrow\, X
\end{equation}
is a vector bundle of rank $d$. The direct image $f_*{\mathcal O}_Y$
of the structure sheaf will be denoted by
$W$. 

We have the natural homomorphism $E\otimes W \,\longrightarrow\, E$ which on any open set 
$U\,\subseteq\, X$ is induced by the bilinear map 
$$
\mathcal{O}_Y(f^{-1}U) \times L(f^{-1}U) \,\longrightarrow\,
L(f^{-1}U),\;\;\; (y,\,l)\,\longmapsto\, y\cdot l\, ,
$$
from the universal property of tensor-product. This homomorphism defines a homomorphism
\begin{equation}\label{e3}
\varphi\, :\, W\, \longrightarrow\, \text{End}(E)\,=\, E\otimes E^*\, ,
\end{equation}
where $E$ is the vector bundle in \eqref{e2}.

The fibers of $\text{End}(E)$ are reductive Lie algebras isomorphic to $\text{Lie}( \text{GL}(d, 
k))\,=\, {\rm M}(d, k)$ (the $d\times d$ matrices with entries in $k$). We recall that a Lie 
subalgebra $A$ of ${\rm M}(d, k)$ is a Cartan subalgebra if
\begin{itemize}
\item $\dim A\,=\, d$, and

\item there is an element $T\, \in\, \text{GL}(d, k)$ such that the conjugation of
${\rm M}(d, k)$ by $T$ takes $A$ into the space of diagonal matrices.
\end{itemize}

\begin{lemma}\label{lem1}
The homomorphism $\varphi$ in \eqref{e3} satisfies the condition that for every closed point $x\, \in\,
X$, the image $\varphi(x)(W_x)$ is a Cartan subalgebra of ${\rm End}(E)_x\,=\, {\rm End}(E_x)$.
\end{lemma}

\begin{proof}
Fix an ordering of the elements of the set $f^{-1}(x)$ of cardinality $d$. Let
$\{y_1,\, \cdots, \, y_d\}$ be this ordered set $f^{-1}(x)$. Fix a nonzero element
$v_i\,\in\, L_{y_i}$ for each $1\, \leq\, i\, \leq\, d$. Since
$$
E_x\,=\,\bigoplus_{i=1}^d L_{y_i}\, ,
$$
the collection $\{v_1,\, \cdots,\, v_d\}$ defines an ordered basis of the
vector space $E_x$. Similarly, we have
$$
W_x\,=\, \bigoplus_{i=1}^d k_i\, ,
$$
where $k_i$ is a copy of $k$. For any $c\, \in\, k_i\,\subset\,
W_x$, the endomorphism $\varphi(c)\,\in\, \text{End}(E_x)$ sends the basis element
\begin{itemize}
\item $v_i$ to $c\cdot v_i$, and

\item $v_j$, $j\, \not=\, i$, to $0$.
\end{itemize}
Therefore, the image $\varphi(W_x)$ is the space of all diagonal matrices with respect
to the above basis $\{v_1,\, \cdots,\, v_d\}$.
\end{proof}

Since $H^0(Y,\, {\mathcal O}_Y)\,=\, H^0(X, \, f_*{\mathcal O}_Y)$, and the
dimension of $H^0(Y,\, {\mathcal O}_Y)$ coincides with the number of connected
components of the scheme $Y$, Lemma \ref{lem1} has the following corollary:

\begin{corollary}\label{cor1}
The number of connected components of the scheme $Y$ coincides with the
dimension of $H^0(X, \, W)$.
\end{corollary}

\subsection{Criterion for direct images under \'etale maps}

Let $F$ and $V$ be vector bundles on $X$ of same rank $d$. Let
\begin{equation}\label{v2}
\varphi\, :\, V\, \longrightarrow\, \text{End}(F)\,=\, F\otimes F^*
\end{equation}
be an ${\mathcal O}_X$--linear homomorphism.

\begin{proposition}\label{prop1}
Assume that $\varphi$ in \eqref{v2} satisfies the condition that for every closed
point $x\, \in\, X$, the image $\varphi(V_x)$ is a Cartan subalgebra of the Lie algebra
${\rm End}(F)_x\,=\, {\rm End}(F_x)$. Then there is an \'etale covering
$$
f\, :\, Y\, \longrightarrow\, X
$$
and a line bundle $L$ over $Y$, such that
$$
f_*L\,=\, F \ \ \ \textit{ and } \ \ \ f_*{\mathcal O}_Y \,=\, V\, .
$$
Furthermore, the homomorphism $\varphi$ coincides with the homomorphism in \eqref{e3}
corresponding to the above triple $(Y,\, f,\, L)$.
\end{proposition}

\begin{proof}
For any closed point $x\, \in\, X$, consider the Cartan subalgebra
$\varphi(V_x)\, \subset\, \text{End}(F_x)$. It produces an \textit{unordered} set of $d$
lines in $F_x$
\begin{equation}\label{uo}
\{l^x_t\}_{t\in B_x}\, , \ \ \ \# B_x\,=\, d
\end{equation}
such that
\begin{itemize}
\item the $d$ lines $\{l^x_t\}_{t\in B_x}$ together generate the fiber $F_x$, and

\item for each $t\, \in\, B_x$,
there is a unique functional $\mu^x_t\, \in\, V^*_x$ with the property that for
all $v\, \in\, V_x$, we have
$$
\varphi(v)(w)\,=\, \begin{cases} \mu^x_t(v)\cdot w ~ \text{ if }~ w\,\in\, l^x_t\\
0 ~ \text{ if }~ w\,\in\, l^x_s,\, s\,\not=\, t\, .
\end{cases}
$$
\end{itemize}
In \eqref{uo} we use the notation $\{l^x_t\}_{t\in B_x}$ instead of $\{l^x_1,\, \cdots, \, 
l^x_d\}$ in order to emphasize that in general these $d$ lines do not have any 
ordering which can be chosen uniformly over $X$. We note that the $d$ elements 
$\{\mu^x_t\}_{t\in B_x}$ of $V^*_x$ are distinct. In fact, $\{\mu^x_t\}_{t\in B_x}$ 
is a basis of the dual vector space $V^*_x$.

The quasiprojective variety defined by the total space of the dual vector
bundle $V^*$ will be denoted by ${\mathbb V}^*$. The locus in ${\mathbb V}^*$ of the above
collection $\{\mu^x_t\}_{t\in B_x, x\in X}$ defines a reduced
subscheme $Y\, \subset\, {\mathbb V}^*$. Let
\begin{equation}\label{Y}
f\, :\, Y\, \longrightarrow\, X
\end{equation}
be the restriction to $Y$ of the natural projection ${\mathbb 
V}^*\,\longrightarrow\, X$. This map $f$ defines an \'etale covering of $X$ of 
degree $d$ because the $d$ lines $\{l^x_t\}_{t\in B_x}$ in \eqref{uo} can be uniformly
ordered over suitable \'etale open subsets of $X$.

Consider the pulled back vector bundle $f^*F\, \longrightarrow\, Y$. It has a line
subbundle $L$ whose fiber over any closed point $\mu^x_t\, \in\, Y$ is the line
$l^x_t$ contained in $F_x$.

The projection formula gives that $f_*f^*F\,=\, F\otimes f_*{\mathcal O}_Y$. This
and the trace homomorphism $f_*{\mathcal O}_Y\, \longrightarrow\,
f_*{\mathcal O}_X$ together give a homomorphism $f_*f^*F\,\, \longrightarrow\, F$. 
This homomorphism and the inclusion map $L\, \hookrightarrow\, f^*F$
together produce a homomorphism
\begin{equation}\label{eta}
\eta\, :\, f_*L \, \longrightarrow\, F\, .
\end{equation}
This $\eta$ is an isomorphism because the $d$ lines $\{l^x_t\}_{t\in B_x}$ in 
\eqref{uo} generate $F_x$.

Now, as constructed in \eqref{e3}, we have a fiberwise injective homomorphism
$$
\widehat{\varphi}\, :\, f_*{\mathcal O}_Y \, \longrightarrow\,
\text{End}(f_*L)\,=\, \text{End}(F)\, .
$$
It is straight-forward to check that the image of $\widehat{\varphi}$ coincides with
the image of $\varphi$. Hence the composition
$(\widehat{\varphi}^{-1}\vert_{\varphi(V)})\circ\varphi$
is an isomorphism from $V$ to $f_*{\mathcal O}_Y$.

In terms of the above isomorphisms $V\, \stackrel{\sim}{\longrightarrow}\, 
f_*{\mathcal O}_Y$ and $\eta$ in \eqref{eta}, the homomorphism $\varphi$ in the statement of
the proposition coincides with the homomorphism in \eqref{e3}.
\end{proof}

Note that the condition in Proposition \ref{prop1} that $\varphi(V_x)$ is a Cartan 
subalgebra of ${\rm End}(F_x)$ for every $x\,\in\, X$ implies that the homomorphism 
$\varphi$ is fiberwise injective.

Combining Lemma \ref{lem1}, Corollary \ref{cor1} and Proposition \ref{prop1} we have the following:

\begin{theorem}\label{thm1}
Given a vector bundle $E$ on $X$ of rank $d$, the following two are equivalent:
\begin{enumerate}
\item There is an \'etale covering
$$
f\, :\, Y\, \longrightarrow\, X\, ,
$$
where $Y$ is a projective scheme, and a line bundle $L$ on $Y$, such that $f_*L\,=\, E$.

\item There is a subbundle ${\mathcal A}\, \subset\, {\rm End}(E)$ of rank $d$ such
that for each closed point
$x\in\, X$, the subspace ${\mathcal A}_x\, \subset\, {\rm End}(E)_x\,=\,
{\rm End}(E_x)$ is a Cartan
subalgebra.
\end{enumerate}
If there is a subbundle ${\mathcal A}\, \subset\, {\rm End}(E)$ as in the second statement,
then $(Y,\, f)$ in the first statement can be so chosen that
${\mathcal A}\,=\, f_*{\mathcal O}_Y$. In that case, the number of connected
components of $Y$ coincides with $\dim H^0(X,\, {\mathcal A})$.
\end{theorem}

\begin{corollary}\label{cor4}
A vector bundle $F$ on $X$ of rank $d$ splits into a direct sum of $d$ line bundles
if and only if there is a trivial subbundle of rank $d$
$$
\iota\, :\, {\mathcal O}_X^{\oplus d}\, \hookrightarrow\, {\rm End}(F)
$$
such that for each closed point
$x\in\, X$, the subspace $${\rm image}(\iota(x))\, \subset\, {\rm End}(F)_x\,=\,
{\rm End}(F_x)$$ is a Cartan subalgebra.
\end{corollary}

\begin{proof}
If $F\, =\, \bigoplus_{i=1}^d L_i$, then the homomorphism
$$
{\mathcal O}_X^{\oplus d}\,=\, \bigoplus_{i=1}^d {\rm End}(L_i)
\, \hookrightarrow\,{\rm End}(F)
$$
satisfies the condition in the statement of the corollary.

To prove the converse, assume that there is a trivial subbundle of rank $d$
$$
\iota\, :\, {\mathcal O}_X^{\oplus d}\, \hookrightarrow\, {\rm End}(F)
$$
satisfying the condition in the corollary. Now the unordered set in \eqref{uo} becomes
uniformed ordered over entire $X$. Therefore, the covering $Y$ in \eqref{Y} becomes
a disjoint union of $d$ copies of $X$. Consequently, the vector bundle $f_*L$ in
\eqref{eta} is a direct sum of line bundles. Since $\eta$ in \eqref{eta} is an
isomorphism, the vector bundle $F$ splits into a direct sum of $d$ line bundles.
\end{proof}

Setting $L\,=\, {\mathcal O}_Y$ in Theorem \ref{thm1} we have the following:

\begin{corollary}\label{cor2}
Given a vector bundle $E$ on $X$, the following two are equivalent:
\begin{enumerate}
\item There is an \'etale covering
$$
f\, :\, Y\, \longrightarrow\, X
$$
such that $f_*{\mathcal O}_Y\,=\, E$.

\item There is a fiberwise injective homomorphism $\alpha\, :\, E\, \longrightarrow\, {\rm End}(E)$
such that for each closed point
$x\in\, X$, the subspace $\alpha(E_x)\, \subset\, {\rm End}(E)_x\,=\, {\rm End}(E_x)$ is a Cartan
subalgebra.
\end{enumerate}
When these statements hold, the number of connected components of the scheme $Y$
coincides with $\dim H^0(X,\, E)$.
\end{corollary}

\begin{remark}
Let $E$ be a vector bundle on $X$ along with two subbundles 
$\mathcal{A},\, \mathcal{B}\,\subseteq\,\mbox{End}(E)$ such that for each closed point 
$x\,\in\, X$, both $\mathcal{A}_x$ and $\mathcal{B}_x$ are Cartan subalgebras of 
$\mbox{End}(E_x)$. Let $f\,:\,Y\, \longrightarrow\, X$ (respectively,
$g\,:\,Z\, \longrightarrow\,
X$) be the \'etale covering and $L\, \longrightarrow\, Y$ (respectively, $L'
\, \longrightarrow\, Z$) be the
line bundle associated to $\mathcal{A}$ (respectively, $\mathcal{B}$). We note 
that if $T$ is an automorphism of $E$ such that $\mathcal{B}\,=\,T^{-1}\mathcal{A}T$, 
then there exists an isomorphism of $X$--schemes $h\,:\,Y\, \longrightarrow\, Z$ such
$h^*L'\,=\, L$. The converse is also true.
\end{remark}

\subsection{An example}

Let $E$ be a vector bundle on $X$ such that there are two Cartan subalgebra
bundles ${\mathcal A}$ and ${\mathcal B}$ of $\text{End}(E)$. It is natural to ask whether
there is an automorphism $T$ of $E$ such that ${\mathcal B}\,=\, T^{-1}{\mathcal A}T$.
We give an example where there is no such $T$.

Let $X$ be a smooth projective elliptic curve defined over $\mathbb C$. It has 
exactly three nontrivial line bundles of order two. Let $L$ and $M$ be two distinct 
nontrivial line bundles on $X$ of order two. Therefore, the third nontrivial line 
bundle of order two is $L\otimes M$.

A theorem of Atiyah says that there are stable vector bundles on $X$ of rank two and 
degree one, and any two of them differ by tensoring with a holomorphic line bundle of 
degree zero \cite[p.~433,Theorem 6]{At}, \cite[p.~434,Theorem 7]{At} (see 
\cite[p.~70, Theorem 4.6]{BB}, \cite[p.~70--71, Theorem 4.7]{BB} for an exposition). 
Take a stable vector bundle $E$ on $X$ of rank two and degree one. If $N$ is a 
holomorphic line bundle on $X$, then
$$
\text{End}(E)\, =\, E\otimes E^*\,=\, (E\otimes N)\otimes
(E\otimes N)^*\,=\, \text{End}(E\otimes N)\, .
$$
Therefore, the endomorphism bundle $\text{End}(E)$ does not depend on the choice of 
the stable vector bundle $E$ of rank two and degree one. 
Since $E$ is stable, the vector bundle $\text{End}(E)$ is polystable. From the
classification of vector bundles on $X$ we know that any polystable vector bundle
on $X$ of degree zero is a direct sum of holomorphic line bundles \cite[p.~433,Theorem 6]{At}
(see also \cite[p.~70, Theorem 4.6]{BB}). It is known that
\begin{equation}\label{t1}
\text{End}(E)\, =\,{\mathcal O}_X\oplus L\oplus M\oplus
(L\otimes M)\, .
\end{equation}
This can also be seen as follows. Let $f\, :\, Y\, \longrightarrow\, X$
be the unramified double cover corresponding to $L$. Then there is a holomorphic
line bundle $\xi$ on $Y$ of degree one such that $E\,=\, f_*\xi$. Since $f_*{\mathcal O}_Y
\,=\, L\oplus {\mathcal O}_X$, this implies that $L$ is a
subbundle of $\text{End}(E)$. Hence $L$ is a direct summand of $\text{End}(E)$
because $\text{End}(E)$ is polystable of degree zero. Similarly,
$M$ and $M\otimes L$ are also direct direct summands of $\text{End}(E)$. Therefore,
it follows that $\text{End}(E)$ decomposes as in \eqref{t1}.

Since $E\,=\, f_*\xi$, we know that $f_*{\mathcal O}_Y\,=\, {\mathcal O}_X\oplus L$
is a Cartan subalgebra bundle of $\text{End}(E)$. Similarly,
${\mathcal O}_X\oplus M$ and ${\mathcal O}_X\oplus (M\otimes L)$ are also
Cartan subalgebra bundles of $\text{End}(E)$.

On the other hand $H^0(X, \, \text{End}(E))\,=\, \mathbb C$ because $E$ is stable;
note that this also follows from \eqref{t1}. Hence the automorphisms of $E$ act
trivially on $\text{End}(E)$.
Therefore, the above Cartan subalgebra bundles ${\mathcal O}_X\oplus L$,
${\mathcal O}_X\oplus M$ and ${\mathcal O}_X\oplus (M\otimes L)$ are not related by
automorphism of $E$.

\section{Ramified coverings of curves}

Throughout this section we assume that $X$ is an irreducible smooth projective curve 
defined over an algebraically closed field $k$.

A \textit{quasiparabolic} structure on a vector bundle $E$ over $X$ consists of the
following:
\begin{itemize}
\item a finite set of reduced distinct closed points $S\,=\, \{x_1,\, \cdots,\, x_n\}\,\subset\,
X$, and

\item for each $x_i\, \in\, S$, a filtration of subspaces
$$0\, \subsetneq\, F^i_1\, \subsetneq\, \cdots \, \subsetneq\,
F^i_{\ell_i}\,=\, E_{x_i}
$$
of the fiber $E_{x_i}$.
\end{itemize}
The subset $S$ is called the \textit{parabolic divisor}. A \textit{quasiparabolic bundle} is
a vector bundle equipped with a quasiparabolic structure. A \textit{parabolic vector bundle} is
a quasiparabolic bundle $(E,\, S,\, \{F^i_j\})$ as above together with real numbers
$\lambda^i_j$, $1\,\leq\, i\, \leq\, n$, $1\,\leq\, j\,\leq\, \ell_i$, such that
$$
1\, > \, \lambda^i_1\, >\, \lambda^i_2\, >\, \cdots\, > \,\lambda^i_{\ell_i} \, \geq\, 0\, .
$$
These numbers $\lambda^i_j$ are called \textit{parabolic weights}. For notational
convenience, a parabolic vector bundle $(E,\, S,\, \{F^i_j\},\, \{\lambda^i_j\})$
is abbreviated as $E_*$. See \cite{MS} for more on parabolic vector bundles.

We will consider only rational parabolic weights. Henceforth, we will assume that all the parabolic
weights are rational numbers.

Let $Y$ be a smooth projective curve, which need not be irreducible, and let
$$
f\,:\, Y\, \longrightarrow\, X
$$
be a finite separable morphism. Let $L_*$ be a parabolic line bundle on $Y$ with $L$
being the underlying line bundle. The direct image $E\, :=\, f_*L$
has a natural parabolic structure which will be described below.

Let $R\, \subset\, Y$ be the set of points where $f$ is ramified. Let $P\, \subset\, Y$
be the parabolic divisor for $L_*$. The parabolic divisor for the parabolic
structure on $E$ is the image $f(R\cup P)$. Take a point $x\, \in\, f(R\cup P)\setminus
f(R)$ in the complement of $f(R)$. Then
$$
(f_* L)_x\,=\, \bigoplus_{y\in f^{-1}(x)} L_y\, .
$$
The quasiparabolic filtration of $(f_* L)_x$ is constructed using this decomposition.
The parabolic weight of the line $L_y\, \subset\, (f_*L)_x$ is the parabolic weight of
$L_*$ at the point $y$. If $y$ is not a parabolic point of $L_*$, then the
parabolic weight of the line $L_y\, \subset\, (f_*L)_x$ is taken to be zero. Combining
these we get a parabolic structure on $E_x$.

Now take any $x\, \in \,f(R)$. Let $\{y_1,\,\cdots,\, y_m\}$ be the reduced
inverse image $f^{-1}(x)_{\rm red}$. The multiplicity of $f$ at $y_i$ will be denoted by $b_i$;
so $f^{-1}(x)\,=\, \sum_{i=1}^m b_iy_i$. For every $1\, \leq\, i\, \leq\, m$, let $V_i\, \subset\,
E_x$ be the image in the fiber $E_x$ for the natural homomorphism
$$
f_* (L\otimes {\mathcal O}_Y(-\sum_{j=1, j\not= i}^m b_jy_j))\, \longrightarrow\, f_*L\,=\, E\, .
$$
We have $\dim V_i\,=\, b_i$, and
\begin{equation}\label{e1}
E_x\,=\, \bigoplus_{i=1}^m V_i\, .
\end{equation}
We will construct a weighted filtration on each $V_i$; these combined together will give the weighted
filtration of $E_x$ using \eqref{e1}. For each $0\,\leq\, \ell\,\leq\, b_i$, let $F^i_\ell\,
\subset\, E_x$ be the image for the natural homomorphism
$$
f_* (L\otimes {\mathcal O}_Y(-\ell y_i -\sum_{j=1, j\not= i}^m b_jy_j))\, \longrightarrow\, f_*L\, .
$$
Note that $F^i_{b_i} \,=\, 0$ and $F^i_0 \,=\, V_i$; in particular, $F^i_\ell\,
\subset\, V_i$ for all $\ell$. It is easy to see that $\dim F^i_\ell\,=\, b_i-\ell$, so
$\{F^i_\ell\}_{\ell=0}^{b_i}$ is a complete flag
of subspaces of $V_i$. The weight of the subspace $F^i_\ell\,\subset\, V_i$,
$0\, \leq\,\ell\, < \,b_i$,
is $(\ell+\lambda_{y_i})/b_i$, where $\lambda_{y_i}$ is the parabolic weight of
$L_*$ at $y_i$; if $y_i$ is not a parabolic point of $L_*$, then $\lambda_{y_i}$
is taken to be zero. Note that $0\, \leq\, (\ell+\lambda_{y_i})/b_i\, <\, 1$.

Now the parabolic structure on $E$ over $x$ is given by these weighted filtrations using
\eqref{e1}. More precisely, for any $0\, \leq\, c\, <\, 1$, if $S^i_x(c)\, \subset\, V_i$,
$1\, \leq\, i\,\leq\, m$, is the subspace of $V_i$ of weight $c$, then the subspace of $E_x$ of
parabolic weight $c$ is the direct sum $\bigoplus_{i=1}^m S^i_x(c)$.

The direct image $f_*L$ equipped with the above parabolic structure will be denoted by
$f_*L_*$.

Since $H^i(Y,\, L)\,=\, H^i(X,\, f_*L)$, using the Riemann--Roch theorem for $L$ and
$f_*L$, we have
$$
\text{degree}(f_*L)\,=\, \text{degree}(L)-\text{genus}(Y)+1+\text{degree}(f)
(\text{genus}(X)-1)\, ,
$$
where $\text{genus}(Y)\,=\, \dim H^1(Y,\, {\mathcal O}_Y)$ (recall that $Y$ need
not be connected). On the other hand,
$$
2(\text{genus}(Y)-1)\,=\, \text{degree}(K_Y)\,=\, \text{degree}(K_X)
+\sum_{y\in R} (b_y-1)\,=\, 2(\text{genus}(X)-1)+\sum_{y\in R} (b_y-1)\, ,
$$
where $b_y$ is the multiplicity of $f$ at $y$ while $K_X$ and $K_Y$ are the
canonical line bundles of $X$ and $Y$ respectively. From these it follows that
$$
\text{par-deg}(f_*L_*)\,=\, \text{par-deg}(L_*)\, .
$$

Generalizing the constructions of direct sum, tensor product and dual of vector 
bundles, there are direct sum, tensor product and dual of parabolic vector bundles 
\cite{Yo}, \cite{MY}, \cite{Bi2}. It should be mentioned that for two parabolic 
vector bundles $E_*$ and $F_*$, while the underlying vector bundle for the parabolic 
direct sum $E_*\bigoplus F_*$ is the direct sum of the vector bundles underlying 
$E_*$ and $F_*$, the underlying vector bundle for the parabolic tensor product 
$E_*\otimes F_*$ is not necessarily the tensor product of the vector bundles 
underlying $E_*$ and $F_*$. Similarly, the underlying vector bundle for the parabolic 
dual $E^*_*$ is different from the dual of the vector bundles underlying $E_*$, 
unless the parabolic structure on $E_*$ is trivial (meaning there are no nonzero 
parabolic weights).

The endomorphism bundle for a parabolic vector bundle $E_*$ is defined to be
\begin{equation}\label{e4}
\text{End}(E_*)\,:=\,E_*\otimes E^*_*\, ;
\end{equation}
it should be emphasized that both the tensor product and dual in \eqref{e4} are
in the parabolic category.

\subsection{When the characteristic is zero}\label{se3.1}

In this subsection we assume that the characteristic of the base field $k$ is zero. 

Let $E_*$ be a parabolic vector bundle on $X$. Let $S\, \subset\, X$ be the parabolic
divisor for $E_*$. The vector bundle underlying $E_*$ will be denoted
by $E_0$. Consider the endomorphism (parabolic) bundle $\text{End}(E_*)$ defined in \eqref{e4}. The
vector bundle underlying it will be denoted by $\text{End}(E_*)_0$. The two vector bundles
$\text{End}(E_*)_0$ and $\text{End}(E_0)$ are identified over the complement $X\setminus S$. This
isomorphism extends to a homomorphism
$$
\beta\, :\, \text{End}(E_*)_0\, \longrightarrow\, \text{End}(E_0)
$$
over entire $X$. For any point $x\, \in\, S$, the subspace $\beta(x)((\text{End}(E_*)_0)_x)
\,\subset\, \text{End}(E_0)_x$ coincides with the space of endomorphisms of the fiber $(E_0)_x$
that preserve the quasiparabolic filtration of $(E_0)_x$. 

A parabolic vector bundle on $X$ can be expressed as the invariant direct image of an equivariant
vector bundle over a (ramified) Galois cover of $X$ \cite{Bi1}, \cite{Bo1}, \cite{Bo2}; recall
the assumption that all the parabolic weights are rational numbers. Let $\widetilde{X}$ be an irreducible
smooth projective curve,
\begin{equation}\label{ga}
\gamma\, :\, \widetilde{X}\, \longrightarrow\, X
\end{equation}
a Galois covering which may be ramified, and $\mathcal E$ a $\Gamma$--linearized 
vector bundle on $\widetilde{X}$, where $\Gamma\, :=\, \text{Gal}(\gamma)$, such that 
$E_*$ corresponds to $\mathcal E$. The vector bundle $E$ underlying $E_*$ is the
invariant direct image $(\gamma_* {\mathcal E})^\Gamma$; note that the action of $\Gamma$ on
${\mathcal E}$ produces an action of $\Gamma$ on $\gamma_* {\mathcal E}$. In particular,
we have $\text{rank}(E)\,=\, \text{rank}({\mathcal E})$. Consider the finite subset
$D'$ of $\widetilde{X}$ consisting of all points $y\, \in\, \widetilde{X}$ satisfying the
following two conditions:
\begin{itemize}
\item $y$ has nontrivial isotropy for the action of $\Gamma$ on $\widetilde{X}$,
and

\item the action of isotropy subgroup $\Gamma_y$ for $y$ on the fiber ${\mathcal E}_y$ is
nontrivial.
\end{itemize}
The image $\gamma(D')\, \subset\, X$ is the subset of $S$ consisting of all points over
which $E_*$ has nontrivial parabolic weight. The above isotropy subgroup $\Gamma_y$
for $y$ is cyclic; let $m_y$ be the order of $\Gamma_y$. Fix a generator
$\nu$ of $\Gamma_y$. A rational number $0\, \leq\, c\, <\, 1$ is a parabolic weight for
$E_*$ at $\gamma(y)$ if and only if $\exp(2\pi\sqrt{-1}c)$ is an eigenvalue for the
action of $\nu$ on ${\mathcal E}_y$. In particular, $c\cdot m_y$ is an integer.

The subbundles of $E_0$ with the parabolic 
structure induced by $E_*$ correspond to subbundles of $\mathcal E$ preserved by the 
action of $\Gamma$. The parabolic vector bundles $E^*_*$ and $\text{End}(E_*)$ 
correspond to ${\mathcal E}^*$ and $\text{End}({\mathcal E})$ respectively; note that 
the $\Gamma$--linearization of $\mathcal E$ induces $\Gamma$--linearizations on both 
${\mathcal E}^*$ and $\text{End}({\mathcal E})$.

A subbundle ${\mathcal A}\, \subset\, \text{End}(E_*)_0$ will be called a \textit{Cartan subalgebra
bundle} if the following two conditions hold:
\begin{enumerate}
\item For each closed point $x\, \in\, X\setminus S$, the fiber ${\mathcal A}_x\, \subset\,
(\text{End}(E_*)_0)_x\,=\, \text{End}((E_0)_x)$ is a Cartan subalgebra of the Lie algebra
$\text{End}((E_0)_x)$.

\item For each point $x\,\in\, S$, the fiber of the $\Gamma$--linearized subbundle of
$\text{End}({\mathcal E})$ corresponding to $\mathcal A$ over a point $y\, \in\, \gamma^{-1}(x)$
is a Cartan subalgebra of the Lie algebra $\text{End}({\mathcal E}_y)$. (If this condition
holds for one point of $\gamma^{-1}(x)$ then it holds for all points of
$\gamma^{-1}(x)$; this is because of the action of $\Gamma$.)
\end{enumerate}

The above definition of a Cartan subalgebra bundle of $\text{End}(E_*)_0$ does not depend on the
choice of the covering $\gamma$. To see this, if
$$
\gamma'\, :\,\widetilde{X}'\, \longrightarrow\, X
$$
is another such covering, then consider the normalization $\mathcal X$ of the fiber product
$\widetilde{X}\times_X \widetilde{X}'$. This covering $\mathcal X$ of $X$ also satisfies the
conditions. If ${\mathcal E}'$ is the equivariant vector bundle on $\widetilde{X}'$ corresponding
to $E_*$, then the pullbacks of $\mathcal E$ and ${\mathcal E}'$ to $\mathcal X$
are equivariantly isomorphic to the equivariant bundle on $\mathcal X$ corresponding
to $E_*$. Hence a fiberwise decomposition of $\mathcal E$ produces a fiberwise decomposition
of its pullback to $\mathcal X$, which in turn descends to a fiberwise decomposition of
${\mathcal E}'$.

Note that ${\mathcal A}\, \subset\, \text{End}(E_*)_0$ is a Cartan subalgebra bundle if and only 
if the $\Gamma$--linearized subbundle $\widetilde{\mathcal A}\,\subset\, \text{End}({\mathcal E})$ 
corresponding to $\mathcal A$ has the property that for every $y\, \in\, \widetilde X$, the 
subspace $\widetilde{\mathcal A}_y\,\subset\, \text{End}({\mathcal E})_y$ is a Cartan subalgebra 
of the Lie algebra $\text{End}({\mathcal E})_y$.

\begin{theorem}\label{thm2}
Given a parabolic vector bundle $E_*$ on $X$, the following two are equivalent:
\begin{enumerate}
\item There is a finite surjective map 
$$
f\, :\, Y\, \longrightarrow\, X\, ,
$$
where $Y$ is a smooth projective curve not necessarily connected,
and a parabolic line bundle $L_*$ on $Y$, such that
\begin{itemize}
\item $E_*$ has a nontrivial parabolic weight at $x\, \in\, X$ if and only if
there is a point $y\, \in\, f^{-1}(x)$ satisfying the condition that either $y$ is a
parabolic point for $L_*$ or $f$ is ramified at $y$ (or both), and

\item the parabolic vector bundle $f_*L_*$ is isomorphic to $E_*$.
\end{itemize}

\item There is a Cartan subalgebra bundle $\mathcal A$ of ${\rm End}(E_*)_0$.
\end{enumerate}
When there is a Cartan subalgebra bundle ${\mathcal A}\, \subset\, {\rm End}(E_*)_0$, the pair
$(Y,\, f)$ in the first statement can be so chosen that the subbundle
${\mathcal A}$ equipped with the parabolic structure induced by
${\rm End}(E_*)$ is isomorphic to $f_*{\mathcal O}_Y$ equipped with the natural parabolic
structure (the parabolic structure on ${\mathcal O}_Y$ is the trivial one, meaning it has
no nonzero parabolic weight). In that case, the number of connected components of $Y$ coincides with $\dim H^0(X,\,
{\mathcal A})$.
\end{theorem}

\begin{proof}
Fix a Galois covering $(\widetilde{X},\, \gamma)$ as in \eqref{ga} with $\Gamma\,=\, 
\text{Gal}(\gamma)$ such that there is a $\Gamma$--linearized vector bundle 
$\mathcal E$ on $\widetilde{X}$ that corresponds to the parabolic vector bundle 
$E_*$.

Assume that the first statement in the theorem holds. Take
$(Y,\, f,\, L_*)$ as in the first statement. Let $\widetilde{Y}$ denote the
normalization of the fiber product $Y\times_X\widetilde{X}$. Let
$$
p_1\, :\, \widetilde{Y}\, \longrightarrow\, Y\ \ \text{ and }\ \
p_2\, :\, \widetilde{Y}\, \longrightarrow\, \widetilde{X}
$$
be the natural projections. We note that the normalization of a fiber product has the
following property: Consider two ramified coverings ${\mathbb A}^1\, \longrightarrow\, {\mathbb A}^1$
defined by $z\, \longmapsto\, z^a$ and $z\, \longmapsto\, z^b$ respectively; then
the projection to the second factor of the normalization of the fiber product
${\mathbb A}^1\times_{{\mathbb A}^1}{\mathbb A}^1$ is an \'etale covering of ${\mathbb A}^1$ if
$b$ is a multiple of $a$. From this it follows that the above projection $p_2$ is an \'etale covering.

The action of $\Gamma$ on $\widetilde{X}$ produces an action
of $\Gamma$ on $Y\times_X\widetilde{X}$, hence $\widetilde{Y}$ is equipped
with an action of $\Gamma$; the map $p_2$ intertwines the actions of $\Gamma$ on
$\widetilde{Y}$ and $\widetilde{X}$. The above morphism $p_1$ is
evidently $\Gamma$ invariant, and hence it produces a morphism
$$
\widetilde{Y}/\Gamma\, \longrightarrow\, Y\, ;
$$
it is easy to see that this morphism is an isomorphism. There is a
$\Gamma$--linearized line bundle $$\widetilde{L}\, \longrightarrow\, \widetilde{Y}$$
that corresponds to the parabolic line bundle $L_*$ on $Y$.

Since $p_2$ is $\Gamma$--equivariant, the action of $\Gamma$ on $\widetilde{L}$ 
produces an action of $\Gamma$ on the direct image $p_{2*}\widetilde{L}$. The 
parabolic vector bundle $f_*L_*$ corresponds to this $\Gamma$--linearized vector 
bundle $p_{2*}\widetilde{L}$.

Since $p_2$ is \'etale, by Lemma \ref{lem1}, there is a homomorphism from the
direct image
$$
\xi\, :\, p_{2*}{\mathcal O}_{\widetilde Y}\, \longrightarrow\,
\text{End}(p_{2*}\widetilde{L})
$$
whose fiberwise images are Cartan subalgebras. The action of $\Gamma$ on $\widetilde 
Y$ produces a $\Gamma$--linearization on ${\mathcal O}_{\widetilde Y}$. Since $p_2$ 
is $\Gamma$--equivariant, the action of $\Gamma$ on ${\mathcal O}_{\widetilde Y}$ 
produces an action of $\Gamma$ on $p_{2*}{\mathcal O}_{\widetilde Y}$. The above 
homomorphism $\xi$ is $\Gamma$--equivariant for the action of $\Gamma$ on 
$\text{End}(p_{2*}\widetilde{L})$ induced by the action of $\Gamma$ on 
$p_{2*}\widetilde{L}$ and the action of $\Gamma$ on $p_{2*}{\mathcal
O}_{\widetilde Y}$. Therefore, the second statement in the theorem holds.

Now assume that the second statement in the theorem holds. Let $\mathcal A$ be a 
Cartan subalgebra bundle of ${\rm End}(E_*)_0$. The $\Gamma$--linearized vector bundle
$\text{End}({\mathcal E})$ corresponds to the parabolic vector bundle ${\rm End}(E_*)$, because
$E_*$ corresponds to $\mathcal E$. Let $\mathcal B$ be the subbundle of
$\text{End}({\mathcal E})$ preserved by the action of $\Gamma$ such that $\mathcal B$ corresponds
to the subbundle $\mathcal A$ of ${\rm End}(E_*)_0$. So for any closed point $x\, \in\, \widetilde X$,
the subspace ${\mathcal B}_x\, \subset\, \text{End}({\mathcal E}_x)$ is a Cartan subalgebra.

By Proposition \ref{prop1}, there is an \'etale covering
$$
\phi\, :\, \widetilde{Y}\, \longrightarrow\, \widetilde{X}
$$
and a line bundle $\mathcal L$ on $\widetilde{Y}$ such that
\begin{equation}\label{is}
\phi_*{\mathcal L}\,=\, {\mathcal E}\, .
\end{equation}
Since $\mathcal B$ is preserved by the action of $\Gamma$ on
$\text{End}({\mathcal E})$ induced by the action of $\Gamma$ on $\mathcal E$, from the construction
in Proposition \ref{prop1} it follows that
\begin{itemize}
\item $\Gamma$ acts on $\widetilde Y$,

\item the map $\phi$ intertwines the actions of $\Gamma$ on $\widetilde Y$ and $\widetilde X$,

\item the line bundle $\mathcal L$ is $\Gamma$--linearized, and

\item the isomorphism in \eqref{is} is $\Gamma$--equivariant.
\end{itemize}
Since $\phi$ is $\Gamma$--equivariant, the composition
$$
\gamma\circ \phi\, :\, {\widetilde Y}\, \longrightarrow\, X
$$
factors through a map
\begin{equation}\label{f}
f\, :\, Y\, :=\, {\widetilde Y}/\Gamma\, \longrightarrow\, X\, .
\end{equation}

Let $q\, :\, {\widetilde Y}\, \longrightarrow\,{\widetilde Y}/\Gamma\,=\,Y$ be the 
quotient map. The parabolic line bundle on $Y$ corresponding to the 
$\Gamma$--linearized line bundle ${\mathcal L}$ on $\widetilde Y$ will be denoted by 
$L_*$.

For any vector bundle $W\, \longrightarrow\, {\widetilde Y}$, there is a canonical
isomorphism
\begin{equation}\label{ci}
(f\circ q)_* W\,=\, f_*q_*W\, \stackrel{\sim}{\longrightarrow}\, \gamma_*\phi_*W
\,=\, (\gamma\circ \phi)_*W\, .
\end{equation}
The action of $\Gamma$ on ${\mathcal L}$ produces actions of $\Gamma$ on both $(f\circ
q)_* {\mathcal L}$ and $(\gamma\circ \phi)_*{\mathcal L}$. The isomorphism in 
\eqref{ci} is $\Gamma$--equivariant for $W\,=\, \mathcal L$.

Since the isomorphism in \eqref{ci} for $\mathcal L$ is $\Gamma$--equivariant, it can 
be deduced that for the above parabolic line bundle $L_*$ on $Y$, the parabolic 
vector bundle $f_*L_*$ on $X$, where $f$ is defined in \eqref{f}, corresponds to the 
$\Gamma$--linearized vector bundle $\phi_*{\mathcal L}$ on $\widetilde X$. Also, as the 
isomorphism in \eqref{is} for ${\mathcal L}$ is $\Gamma$--equivariant, and the parabolic
vector bundle $E_*$ corresponds to the $\Gamma$--linearized vector bundle $\mathcal E$, from the 
above observation on \eqref{ci} it also follows that the two parabolic vector bundle 
$f_*L_*$ and $E_*$ are isomorphic.

In view of the above proof, from Theorem \ref{thm1} we conclude the following: If 
there is a Cartan subalgebra bundle ${\mathcal A}\, \subset\, {\rm End}(E_*)_0$, 
then the pair $(Y,\, f)$ in the first statement of the theorem can be so chosen that 
the subbundle ${\mathcal A}$ equipped with the parabolic structure induced by ${\rm 
End}(E_*)$ is isomorphic to $f_*{\mathcal O}_Y$ equipped with the natural parabolic 
structure (the parabolic structure on ${\mathcal O}_Y$ is the trivial one). In that 
case, the number of connected components of $Y$ coincides with $\dim H^0(X,\, 
{\mathcal A})$.
\end{proof}

In Theorem \ref{thm2}, setting $L\,=\, {\mathcal O}_Y$ equipped with
the trivial parabolic structure, we have the following:

\begin{corollary}\label{cor3}
Given a parabolic vector bundle $E_*$ on $X$, the following two are equivalent:
\begin{enumerate}
\item There is a finite surjective map 
$$
f\, :\, Y\, \longrightarrow\, X\, ,
$$
where $Y$ is a smooth projective curve not necessarily connected, such that 
$f_*{\mathcal O}_Y$ equipped with the natural parabolic structure is
isomorphic to $E_*$ (the parabolic structure on ${\mathcal O}_Y$ is the trivial one).

\item There is a homomorphism of parabolic bundles $\alpha\, :\, E_*\, \longrightarrow\,
{\rm End}(E_*)$ such that
\begin{itemize}

\item $\alpha$ is an isomorphism of $E_*$ with the image $\alpha(E_*)$ equipped
with the parabolic structure induced by the parabolic structure of ${\rm End}(E_*)$, and

\item $\alpha(E_0)$ is a Cartan subalgebra bundle of ${\rm End}(E_*)_0$, where $E_0$ is
the vector bundle underlying $E_*$.
\end{itemize}
\end{enumerate}
When these hold, the number of connected components of $Y$ coincides with $\dim H^0(X,\,
E_0)$.
\end{corollary}

\subsection{The case of positive characteristic}\label{sc3p}

Assume that the base field $k$ is of positive characteristic. Let $p$ denote the 
characteristic of $k$.

The correspondence between parabolic vector bundles and equivariant vector bundles used 
extensively in Section \ref{se3.1} remains valid under the following tameness condition (see 
\cite{Bo1}, \cite{Bo2}):

If $a/b$ is a parabolic weight, where $a$ and $b$ are nonzero coprime integers, then we assume
that $b$ is not a multiple of $p$.

Once we impose the above condition on $E_*$, the proof of Theorem \ref{thm2} goes through without
any change. Similarly, Corollary \ref{cor3} remains valid after the above tameness condition on
$E_*$ is imposed.

\begin{remark}\label{rem1}
Theorem \ref{thm1} and Corollary \ref{cor2} remain valid when $X$ is a root-stack; see
\cite{Ca}, \cite{Bo1}, \cite{Bo2} for root-stacks. When the parabolic divisor is a simple
normal crossing divisor, quasiparabolic filtrations satisfy certain conditions and the
parabolic weights are tame (in positive characteristic case), there is an equivalence between
parabolic vector bundles over a smooth projective variety $Y$ and vector bundles over smooth
root-stack whose underlying coarse moduli space is $Y$; see \cite{Bo1}, \cite{Bo2}. Therefore,
under the above assumptions on parabolic structure, Theorem \ref{thm2} and Corollary
\ref{cor3} extend to higher dimensions.
\end{remark}

\section{Factoring of covering maps}\label{se5}

In this section we assume that the characteristic of $k$ is zero.

Let 
\begin{equation}\label{factor}\xymatrix{Y\ar[rr]^f\ar[dr]_g&&X\\&Z\ar[ur]_h}\end{equation}
be a commutative diagram of \'etale coverings, and define $V:=h_*\mathcal{O}_Z$ and 
$W:=f_*\mathcal{O}_Y$. Notice that the homomorphism $\mathcal{O}_Z\,
\longrightarrow\, g_*\mathcal{O}_Y$ induces 
a homomorphism
\begin{equation}\label{dc}
V\,\longrightarrow\, W
\end{equation}
which is fiberwise injective, and so $V$ is a subbundle of $W$. Notice moreover 
that $V$ gives way to a Cartan subalgebra of $\mbox{End}(V)$ and $W$ produces a Cartan subalgebra 
of $\mbox{End}(W)$. Let $\mathbb{V}^*$ and $\mathbb{W}^*$ be the total spaces of $V^*$ and $W^*$, 
respectively. We have the following commutative diagram\\

\centerline{\xymatrix{Y\ar[dr]_g\ar@{^{(}->}[r]&(g_*\mathcal{O}_Y)^*\ar@{^{(}->}[r]\ar[d]&\mathbb{W}^*\ar[d]\\&Z\ar[dr]_h\ar@{^{(}->}[r]&\mathbb{V}^*\ar[d]\\&&X}}
 
\vspace{0.4cm}

\noindent where the map $\mathbb{W}^*\,\longrightarrow\, \mathbb{V}^*$ is the dual of the homomorphism
in \eqref{dc}, and $(g_*\mathcal{O}_Y)^*$ is the restriction of this fiber bundle $\mathbb{W}^*$ to
$Z\, \subset\, \mathbb{V}^*$. Over $x\in X$, in $\mathbb{W}^*$ we have the functionals
$\mathcal{G}_x\,:=\,\{\mu_{t,W}^x\}_{t\in B_{x,W}}$ that define the preimage of $x$ in $Y$ and in $\mathbb{V}^*$ we have the linear functionals $\mathcal{H}_x:=\{\mu_{t,V}^x\}_{t\in B_{x,V}}$ that define the preimage of $x$ in $Z$. We note that under the map $\mathbb{W}^*\,\longrightarrow\, \mathbb{V}^*$, $\mathcal{G}_x$ is taken to $\mathcal{H}_x$. This implies that for every $v\in V$, there exists $t'\in B_{x,W}$ such that for every $v'\in\ell_{t,V}^x$,
$$v(v')\,=\,\mu_{t',W}^x(v)\cdot v'.$$
Now notice that $V$ is a direct summand of $W$, and therefore we have homomorphisms $i_V
\,:\,V\,\longrightarrow\, W$
and $p_V\,:\,W\,\longrightarrow\, V$ such that $p_Vi_V\,=\,\mbox{id}_V$. These induce a homomorphism
$\psi\,:\,\mbox{End}(W)\,\longrightarrow\,\mbox{End}(V)$, $f\,\longmapsto\, p_{V}\circ f\circ i_V$. Now, the previous condition just means that if $V_x$ is embedded in $\mbox{End}(V_x)$ as a Cartan subalgebra, $W_x$ is embedded in $\mbox{End}(W_x)$ as a Cartan subalgebra, then the following diagram commutes:\\

\centerline{\xymatrix{V_x\ar@{^{(}->}[r]^{i_{V,x}}\ar@{^{(}->}[d]&W_x\ar@{^{(}->}[d]\\\mbox{End}(V_x)&\mbox{End}(W_x)\ar[l]^{\psi_x}}}

Indeed, we see that necessarily there is a subset of $\{\ell_t^x\}_{t\in B_{W,x}}$ that generates $V_x$, and so the commutativity of the previous diagram just means that $\mu_{t,W}^x$ is taken to a $\mu_{t',V}^x$.
 
By retracing our steps, we have proved the following proposition:

\begin{proposition}
Let $f\,:\,Y\,\longrightarrow\, X$ be an \'etale covering. Then there is a bijection between intermediate \'etale 
coverings as in (\ref{factor}) and direct summands $V$ of $f_*\mathcal{O}_Y$ such that $V_x$ has 
an embedding as a Cartan subalgebra of ${\rm End}(V_x)$ for every $x\in X$ and such that the 
induced diagram
\centerline{\xymatrix{V_x\ar@{^{(}->}[r]\ar@{^{(}->}[d]&(f_*\mathcal{O}_Y)_x\ar@{^{(}->}[d]\\
{\rm End}(V_x)&{\rm End}((f_*\mathcal{O}_Y)_x)\ar[l]}}
commutes.
\end{proposition} 

\section*{Acknowledgements}

We thank the referee for helpful comments to improve the exposition. The first author is partially supported by Fondecyt Grant 3150171 and CONICYT PIA ACT1415. The second author 
wishes to thank the Universidad de Chile for hospitality while the work was carried 
out. He is partially supported by a J. C. Bose Fellowship.



\begin{thebibliography}{AAAA}

\bibitem[At]{At} M. F. Atiyah, Vector bundles over an elliptic curve, \textit{Proc. 
London Math. Soc.} \textbf{7} (1957) 414--452.

\bibitem[BB]{BB} U. N. Bhosle and I. Biswas, Notes on vector bundles on curves,
\textit{Teichm\"uller theory and moduli problem}, 61--93, Ramanujan Math. Soc.
Lect. Notes Ser., 10, Ramanujan Math. Soc., Mysore, 2010.

\bibitem[Bi1]{Bi1} I. Biswas, Parabolic bundles as
orbifold bundles, \textit{Duke Math. Jour.}
\textbf{88} (1997) 305--325.

\bibitem[Bi2]{Bi2} I. Biswas, Parabolic ample bundles,
\textit{Math. Ann.} \textbf{307} (1997) 511--529.

\bibitem[Bo1]{Bo1} N. Borne, Fibr\'es paraboliques et champ des racines,
\textit{Int. Math. Res. Not.} (2007), no. 16, Art. ID rnm049.

\bibitem[Bo2]{Bo2} N. Borne, Sur les repr\'esentations du groupe
fondamental d'une vari\'et\'e prive\'e d'un diviseur \`a croisements
normaux simples, \textit{Indiana Univ. Math. Jour.} \textbf{58} (2009),
137--180.

\bibitem[Ca]{Ca} C. Cadman, Using stacks to impose tangency conditions on curves, {\it Amer.
Jour. Math.} {\bf 129} (2007), 405--427.

\bibitem[DP]{DP} A. Deopurkar and A. Patel, Vector bundles and finite covers,
arXiv:1608.01711 [math.AG].

\bibitem[MY]{MY} M. Maruyama and K. Yokogawa, Moduli of
parabolic stable sheaves, \textit{Math. Ann.} \textbf{293}
(1992) 77--99.

\bibitem[MS]{MS} V. B. Mehta and C. S. Seshadri, Moduli of vector bundles on curves
with parabolic structure, {\it Math. Ann.} \textbf{248} (1980), 205--239.

\bibitem[Yo]{Yo} K. Yokogawa, Infinitesimal deformations of
parabolic Higgs sheaves, \textit{Int. Jour. Math.}
\textbf{6} (1995) 125--148.

\end{thebibliography}
\end{document}